\documentclass[11pt]{article}
 \usepackage{graphicx}
 \usepackage{amssymb}
 \usepackage{epstopdf}
 \usepackage{pifont}
 \usepackage{amsmath}
 \usepackage{enumerate}
 \usepackage[sans]{dsfont}
 \usepackage{graphicx}
 \usepackage{stmaryrd}
 \usepackage{amsthm}
 \usepackage{verbatim}
 \usepackage[all]{xypic}
 \usepackage[mathscr]{euscript}
 \usepackage{mystyle}
 \usepackage{verbatim}

 \usepackage{ccfonts}

 \newtheoremstyle{mythmsty}{15pt}{5pt}{\it}{0pt}{\textbf}{}{0pt}{\textbf{(\thmnumber{#2})} \thmname{#1} \thmnote{ #3}
}

 \theoremstyle{mythmsty}

 \newtheorem{thm}[paragraph]{\bf Theorem}
 \newtheorem{cor}[paragraph]{\bf Corollary}
 
 \newtheorem{prop}[paragraph]{\bf Proposition}
 \newtheorem{lemma}[paragraph]{\bf Lemma}

 \newenvironment{thm*}{\bf Theorem \it}{}
 \newenvironment{cor*}{\bf Corollary \it}{}
 \newenvironment{definition*}{\bf Definition \it}{}
 \newenvironment{prop*}{\bf Proposition \it}{}
 \newenvironment{lemma*}{\bf Lemma \it}{}
 \newenvironment{fact*}{\bf Fact \it}{}
 \newenvironment{rmk*}{\bf Remark \it}{}
 \newenvironment{exercise*}{\bf Exercise \it}{}
 \newenvironment{solution*}{\underline{Solution}\quad\rm}

  \numberwithin{equation}{subsubsection}

\newcommand{\mysec}[1]{\noindent \section{#1} \label{#1}}
\newcommand{\mysubsec}[1]{\noindent \subsection{#1} \label{#1}}
\newcommand{\ppp}[1]{\noindent\subsubsection{}\label{#1}\hskip -.1455in\textbf{)}}
\newcommand{\rrf}[1]{(\ref{#1})}

\title{\bf  Derived equivalence and birationality}
\author{{\small \textsc{Yu-Han Liu}}}
\begin{document}
\date{}
\maketitle

\setcounter{tocdepth}{2}

\setcounter{section}{0}

\mysec{Introduction}

\mysubsec{Introduction}

\ppp{}  Two smooth projective varieties are called derived equivalent, or D-equivalent, if their (bounded) derived categories of coherent sheaves are exact equivalent.  Kawamata formulated conjectures relating D-equivalence with K-equivalence \cite{kawamata_dk}.  Some partial answers to these conjectures have been known; see \cite[Section 6.4]{huybrechts_fm} and \cite{kawamata_kd}. 

If two varieties are K-equivalent then they are birational, and there are examples of D-equivalent varieties which are not birational: most well-known among them perhaps abelian varieties and their duals of dimension at least two.  In this paper we prove a sufficient condition \rrf{thm:bir} when D-equivalence does imply birationality.  

The form \rrf{cor:bir} in which we will apply the criterion was already known, as the author discovered after having written up the paper.  The slight shift of emphasis from geometry to tensor functors will be justified only in future work.

\ppp{}  Since birationality is equivalent to isomophy for smooth projective curves, we use this birational criterion in \rrf{thm:curves} to show that over an algbraically closed field (of \emph{any} characteristic), D-equivalent curves are isomorphic.  The cases when the curve genus is zero or at least 2 can be reduced to a more general result due to Bondal-Orlov \cite[Theorem 2.5]{BO_reconstruction} about varieties with ample (anti-)canonical bundles; our approach gives also a proof of slightly different flavor for these cases. 

The most interesting case is of elliptic curves.  The proof presented in this paper has the advantage of being algebraic and does not use the complex topology.  It is perhaps also closer in spirit to the first known examples of D-equivalent varieties:  They should be considered as moduli spaces of each other, but for an elliptic curve all the ``interesting'' moduli spaces are essentially isomorphic to the curve itself; \cite[Theorem 7]{atiyah_ellipticvectorbundles} and \cite[Theorem 1]{tu_elliptic}.

\ppp{}  Kawamata showed \cite[Theorem 2.3]{kawamata_dk} that for smooth projective varieties of maximal Kodaira dimensional, D-equivalence implies K-equivalence.  In particular in this case D-equivalence implies birationality.  In \rrf{cor:DK} we prove this last implication as another application.  Our approach is to retrieve information on the rational maps associated to the linear system of pluri-canonical forms through the Serre functors on the derived categories.

\ppp{}  \emph{Notations.}  For any variety we denote by $T_X:=D(X)_\mathrm{parf}$ the triangulated category of perfect complexes.  When $X$ is smooth $T_X$ is isomorphic to the derived category $D^b(X)$ of coherent sheaves on $X$.  If $k$ is a field we have $T_k=D^b(\mathrm{Vect}_k)$.

All functors considered in this paper are derived, and, for example, we denote the derived pull-back $Rf^*$ (resp. derived push-forward $Rf_*$) of a morphism $f$ simply by $f^*$ (resp. $f_*$).

In the case when a $k$-linear triangulated category $T$ is $\mathrm{Hom}$-finite, for any pair of objects $a,b\in T$ we denote by $\mathrm{Hom}^\bullet(a,b)$ the object $\bigoplus_j\mathrm{Hom}_T(a,b[j])[-j]$ in $T_k$.  The $\mathrm{Hom}$-finite condition is satisfied, for example, when $k$ is a field and $T=T_X$ where $X$ is a smooth projective variety over $k$.

\mysec{A criterion for birationality}

\mysubsec{Varieties sharing a point}

\ppp{}  Let $X$ be a variety over a field $k$; denote by $\pi$ the structural morphism $X\rightarrow \mathrm{Spec}(k)$.  We will need the following 

\begin{lemma}\label{lem:supp}  If $p$ is an object in $T_X$ such that the functor $T_X\rightarrow T_k$ defined by $a\mapsto \pi_*(a\otimes p)$ is a tensor functor, then $p$ is supported at a closed $k$-point $x$ on $X$.


\end{lemma}

\begin{proof}  By \cite[Theorem (3.3.4.1)]{liu_TTsp}, the tensor functor $a\mapsto \pi_*(a\otimes p)$ is isomorphic to $a\mapsto x^*(a)$ for some closed $k$-point $x\in X$; since $x^*(k(x))$ is non-zero, the case of $a=k(x)$ implies that $x$ lies in the cohomology support $\mathrm{supp}(p)$ of $p$.  

Suppose there is another closed $k$-point $y\neq x$ in $\mathrm{supp}(p)$, then we have $x^*(k(y))=0$.  On the other hand, $\pi_*(k(y)\otimes p)\cong \pi_*y_*y^*(p)$ is non-zero, since $y^*(p)$ is non-zero and $\pi_*y_*: T_{k(y)}\rightarrow T_k$ is an equivalence.  This is a contradiction and we must have $\mathrm{supp}(p)=\{x\}$.    \end{proof}

\ppp{}  Let $X$ and $Y$ be two smooth projective varieties over a field $k$.  

\begin{thm}\label{thm:bir}  Suppose $\Phi: T_X\rightarrow T_Y$ is an exact equivalence, and $F: T_X\rightarrow T_k$ and $G: T_Y\rightarrow T_k$ are tensor functors such that the following diagram is commutative up to functor isomorphisms:

\centerline{\xymatrix{T_X \ar[rr]^-\Phi \ar[dr]_-F&& T_Y \ar[dl]^-G\\ & T_k. }}

Then $X$ and $Y$ are birational.   \end{thm}

\begin{proof}  Again by \cite[Theorem (3.3.4.1)]{liu_TTsp} we may replace $F$ with $x^*$ and $G$ with $y^*$, where $x\in X$ and $y\in Y$ are closed $k$-points.  By \cite[Theorem 3.2.2]{orlov_dercat} the equivalence $\Phi$ is isomorphic to the Fourier-Mukai transformation $\Phi_\mathcal E$ for some object $\mathcal E\in T_{X\times Y}$; let $\mathcal E_y$ be the restriction of $\mathcal E$ to $X\times\{y\}\hookrightarrow X\times Y$.  Now we have the following diagram:

\centerline{\xymatrix{& X\times Y \ar[dl]_-{\pi_X} \ar[d]^-{\pi_Y} & X\times\{y\} \ar[l] \ar[d]^-\pi \\ X & Y & y. \ar[l]}}

Consider the functor $y^*\circ \Phi_{\mathcal E}$, which by \cite[Lemma 1.3]{bo_sod} is isomorphic to the Fourier-Mukai transformation $\Phi_{\mathcal E_y}: a\mapsto \pi_*(a\otimes \mathcal E_y)$.  By assumption this is isomorphic to the tensor functor $a\mapsto x^*(a)$, and so by \rrf{lem:supp} we have $\mathrm{supp}(\mathcal E_y)=\{x\}$ on $X\times\{y\}\cong X$.  By \cite[Lemma 3.3(b)]{thomason_class} we then have $\mathrm{supp}(\mathcal E)\cap (X\times\{y\})=\mathrm{supp}(\mathcal E_y)=\{x\}$.

Now the projection $\mathrm{supp}(\mathcal E)\rightarrow Y$ is surjective with connected fibres \cite[Lemma 6.4 and Lemma 6.11]{huybrechts_fm}, whose dimensions are upper-semicontinuous.  Hence there is an open neighbourhood $V$ of $y$ in $Y$ over which the morphism $\mathrm{supp}(\mathcal E)\rightarrow Y$ is a homeomorphism.  

In particular the support of $\Phi(\mathcal O_X)= \pi_{Y*}(\mathcal E)$ contains $V$; its fibre at $y$ is isomorphic to $k(y)$.  Therefore by shrinking $V$ if necessary, we have $\Phi(\mathcal O_X)|_V\cong \mathcal O_V$.

Now consider the following commutative diagram:

\centerline{\xymatrix{T_X \ar[rr]^-\Phi \ar[d]^-{q_x}  && T_Y \ar[d]^-{q_y} \ar[r]^-{|_V} & T_V \ar[d]^-{q'_y} \\ T_X/\ker(x^*) \ar[rr]^-{\bar\Phi} \ar[dr]_-{\bar x^*}&& T_Y/\ker(y^*) \ar[dl]^-{\bar y^*} \ar[r]_-{\mathrm{f.f.}} & T_V/\ker(y^*)  \\ & T_k, }}

where $\bar \Phi$ is the induced exact equivalence, and the arrow marked with ``f.f.'' is fully faithful.  We then have $\bar\Phi q_x(\mathcal O_X)=q_y\Phi(\mathcal O_X)$, whose image in $T_V/\ker(y^*)$ is the unit object $q'_y(\mathcal O_V)$.  

In particular, under the composition \[T_X/\ker(x^*)\longrightarrow T_Y/\ker(y^*)\longrightarrow T_V/\ker(y^*)\] the unit object $q_x(\mathcal O_X)$ is sent to the unit object $q'_y(\mathcal O_V)$.  This composition is fully faithful, and so we have an isomorphism \[\mathrm{End}_{T_X/\ker(x^*)}(q_x(\mathcal O_X))\cong \mathrm{End}_{T_V/\ker(y^*)}(q'_y(\mathcal O_V)).\]  But by \cite[Theorem (3.3.4.1) and Proposition (2.3.3.1)]{liu_TTsp} these two are isomorphic to respectively the local rings $\mathcal O_{X,x}$ and $\mathcal O_{V,y}\cong \mathcal O_{Y,y}$, and so $X$ and $Y$ are birational.  \end{proof}

\begin{cor}\label{cor:bir}\cite[Lemma 2.5]{bm_complexsurfaces}  Let $X$ and $Y$ be smooth projective varieties over a field $k$; let $x\in X$ and $y\in Y$ be closed $k$-points.  Suppose $\Phi: T_X\rightarrow T_Y$ is an exact equivalence sending $k(x)$ to a shift of $k(y)$, then $X$ and $Y$ are birational.
\end{cor}

\begin{proof}  By pre- and post-composition with the shift functors on $T_X$ and $T_Y$, we may assume that $\Phi$ sends $k(x)[-\dim(X)]$ to $k(y)[-\dim(Y)]$.  Denote by $F: T_X\rightarrow T_k$ the functor represented by $k(x)[-\dim(X)]$, in the sense that $F(a):=\mathrm{Hom}^\bullet(k(x)[-\dim(X)],a)$ for every $a\in T_X$.  Similarly let $G$ be represented by $k(y)[-\dim(Y)]$, then we have $F\cong G\circ\Phi$.

But then we are in the situation of \rrf{thm:bir} since these functors are isomorphic to the tensor functors $x^*$ and $y^*$ by Grothendieck-Verdier duality \cite[Corollary 3.35]{huybrechts_fm}:  For any $a\in T_X$ we have \[\mathrm{Hom}^\bullet_{T_X}(k(x)[-\dim(X)],a)=\mathrm{Hom}^\bullet_{T_X}(x_*(k), a[\dim(X)])\cong \mathrm{Hom}^\bullet_{T_k}(k,x^*(a))\cong x^*(a)\in T_k.\]  \end{proof}

\begin{cor}\label{cor:bir2}  Let $X$ and $Y$ be smooth projective varieties over a field $k$; let $y\in Y$ be a closed $k$-point.  Suppose $\Phi: T_X\rightarrow T_Y$ is an exact equivalence sending an object $v$ to a shift of $k(y)$.  If the cohomology support $\mathrm{supp}(v)$ is zero-dimensional, then $X$ and $Y$ are birational.
\end{cor}

\begin{proof}  Suppose $\Phi(v)\cong k(y)[n]$, then we have \[\mathrm{Hom}(v,v)\cong \mathrm{Hom}(k(y), k(y))\cong k.\]  Moreover, \[\mathrm{Hom}(v,v[i])\cong \mathrm{Hom}(k(y), k(y)[i])=0\]for all $i<0$.  See for example \cite[Proposition 11.8]{huybrechts_fm}.

Thus the conditions of \cite[Lemma 4.5]{huybrechts_fm} are satisfied, and we conclude that $v$ is isomorphic to a shift of $k(x)$ for some closed $k$-point $x\in X$, and we are in the situation of \rrf{cor:bir}.   \end{proof}

\mysec{Applications}

\mysubsec{Smooth projective curves with equivalent derived categories}

\ppp{}  We consider the following

\begin{thm}\label{thm:curves}  Let $k$ be an algebraically closed field.  If $C_1$ and $C_2$ are two smooth projective curves admitting an exact equivalence $\Phi:T_{C_1}\longrightarrow T_{C_2}$, then $C_1$ and $C_2$ are isomorphic.  \end{thm}

In fact when the genus $g$ of $C_1$ is either zero or at least two the result holds over any field due to the more general result \cite[Theorem 2.5]{BO_reconstruction}; see also \cite[Theorem (4.2.1.1)]{liu_TTsp} for a somewhat different proof.  (In \rrf{par:rmk} below we sketch an independent proof in the case of curves.)

We are thus reduced to the special case:

\begin{prop}\label{prop:ell}  Let $k$ be an algebraically closed field.  If $E_1$ and $E_2$ are two smooth projective curves of genus one admitting an exact equivalence $\Phi:T_{E_1}\longrightarrow T_{E_2}$, then $E_1$ and $E_2$ are isomorphic.  \end{prop}

The case when $k=\mathbb C$ can be proved by passing to singular cohomology \cite[Corollary 5.46]{huybrechts_fm}, but this proof does not seem to trivially generalize to other fields.  It may be an interesting question to ask whether \rrf{prop:ell} still holds over fields which are not necessarily algebraically closed.

The rest of this section except for \rrf{par:rmk} is devoted to giving the proof of \rrf{prop:ell}, and we work under its notations and assumptions.  

\ppp{}  Let $v$ be an object in $T_{E_1}$ such that $\Phi(v)\cong k(y)[-1]$ for a closed $k$-point $y\in E_2$.  Recall that $k(y)[-1]$ represents the functor $y^*: T_{E_2}\rightarrow T_{k(y)}$; see the proof of \rrf{cor:bir}.  Let $F: T_{E_1}\rightarrow T_k$ be the exact functor represented by $v$: \[F(b):=\mathrm{Hom}^\bullet_{T_{E_1}}(v,b)\in T_k\] for every $b\in T_{E_1}$.  Then for any $b\in T_{E_1}$ we have \[F(b):=\mathrm{Hom}^\bullet_{T_{E_1}}(v,b)\cong \mathrm{Hom}^\bullet_{T_{E_2}}(\Phi(v),\Phi(b))=y^*\Phi(b).\]

Hence we have a commutative diagram

\centerline{\xymatrix{T_{E_1} \ar[rr]^-\Phi \ar[dr]_-F && T_{E_2} \ar[dl]^-{y^*}\\ & T_k. }}

\begin{prop}\label{prop:obs}  With notations as above, we have:\begin{enumerate}[(a)]
\item $\mathrm{Hom}(v,v)=k$.
\item The objects $v$ is isomorphic either to a shift of $k(x)$ for some point $x\in E_1$ or to a shift of a vector bundle on $E_1$.  
\end{enumerate}
\end{prop}

\begin{proof}  (a)  It follows from the fact that $\mathrm{Hom}(v,v)\cong \mathrm{Hom}(k(y), k(y))\cong k$ since $v\mapsto k(y)[-1]$ under the equivalence $\Phi$.

(b)  By \cite[Corollary 3.15]{huybrechts_fm} any object in $T_{E_1}$ is isomorphic to a direct sum of shifts of coherent sheaves on $E_1$.  By part (a) there must be only one summand, hence $u$ is the shift of a coherent sheaf.  Since any coherent sheaf on a smooth curve is also the direct sum of its torsion subsheaf with a locally free sheaf, we see that $v$ is the shift of either a torsion sheaf supported at a closed $k$-point or a vector bundle.  In the torsion case it must also be reduced since otherwise $\mathrm{Hom}(v,v)$ would have dimension greater than one.  \end{proof}

\ppp{}  The proof of \rrf{prop:ell} will proceed by induction on the rank of (the shift of) the coherent sheaf $v$; by this we mean the rank of the unique (up to isomorphism) coherent sheaf isomorphic to $v$ in $T_{E_1}$.  The base cases are the following:

\begin{lemma}\label{lem:reduction}  If the rank of $v$ is $0$ or $1$ then $E_1$ and $E_2$ are isomorphic.  \end{lemma}

\begin{proof}  By \rrf{prop:obs}, if $v$ is torsion then it is a shift of $k(x)$.  Then $E_1$ and $E_2$ are birational by \rrf{cor:bir}, hence isomorphic since they are assumed to be both smooth.

On the other hand, assume that $v$ is a shift of a line bundle, then $v\cong L[n]$ for some $n\in\mathbb Z$ and line bundle $L$ on $E_1$.  Fix a closed $k$-point $z\in E_1$, consider exact equivalences \[T_{\mathrm{Pic}^0(E_1)}\stackrel{\Phi_{\mathcal P}}{\longrightarrow} T_{E_1} \stackrel{\Phi_\mathcal E}{\longrightarrow} T_{E_1},\]where $\mathcal P$ is the Poincar\'e line bundle on $\mathrm{Pic}^0(E_1)\times E_1$, and $\mathcal E=\Delta_*\mathcal O_{E_1}(\deg(L)z)[n]$ on $E_1\times E_1$; here $\Delta: E_1\longrightarrow E_1\times E_1$ denotes the diagonal morphism.

With our choice of Fourier-Mukai kernel $\mathcal E$ we have a line bundle $M$ of degree zero on $E_1$ such that $\Phi_\mathcal E(M)=v$; moreover, there is a closed $k$-point $m\in \mathrm{Pic}^0(E_1)$ such that $\Phi_\mathcal P(k(m))=M$.  Hence the composition $\Phi\circ \Phi_\mathcal E\circ\Phi_\mathcal P$ is an exact equivalence sending $k(x)$ to $k(y)[-1]$, and we are in the situation treated above:

\centerline{\xymatrix{ T_{\mathrm{Pic}^0(E_1)}\ar[r]^-{\Phi_\mathcal P} \ar[rrd]_-{}& T_{E_1} \ar[r]^-{\Phi_\mathcal E}&T_{E_1} \ar[r]^-\Phi & T_{E_2} \ar[dl]^-{y^*}\\ && T_k. }}

Hence we have $E_2\cong\mathrm{Pic}^0(E_1)\cong E_1$ by \rrf{cor:bir}.  \end{proof}

\ppp{}  The proof of \rrf{prop:ell} is now reduced to the following

\begin{lemma}  With notations and assumptions as in \rrf{prop:ell}, if $v$ is of rank at least two then there exists another exact equivalence $\Phi': T_{E_1}\rightarrow T_{E_2}$ and object $v'\in T_{E_1}$ such that $\Phi'(v')\cong k(y)[-1]$ and $\mathrm{rank}(v')<\mathrm{rank}(v)$.  \end{lemma}

\begin{proof}  Let $r=\mathrm{rank}(v)\geq 2$.  By tensoring with a suitable line bundle on $E_1$ we can write $v\cong L\otimes B[n],$ where $n\in\mathbb Z$, $L$ is a line bundle, and $B$ an indecomposable vector bundle of rank $r$ and degree $d_B$ satisfying $0\leq d_B<r$.  

We claim that $d_B$ cannot be zero:  Indeed, by \cite[Theorem 5 (ii)]{atiyah_ellipticvectorbundles} every indecomposable vector bundle of degree zero and rank $r$ on $E_1$ is of the form $M\otimes F_r$, where $M$ is a line bundle of degree zero, and $F_r$ is \emph{the} special vector bundle of degree zero and rank $r$; see \cite[Theorem 5 (i)]{atiyah_ellipticvectorbundles}.  Hence if $d_B=0$ then we have \[\mathrm{Hom}(v,v)\cong \mathrm{Hom}(L\otimes M\otimes F_r[n],L\otimes M\otimes F_r[n])\cong\mathrm{Hom}(F_r,F_r)\cong k^r\] by \cite[Corollary 1 to thm 5, and Lemma 17]{atiyah_ellipticvectorbundles}.  This contradicts \rrf{prop:obs} part (a) since $r\geq 2$.


Now consider the Poincar\'e line bundle $\mathcal P$ on $\mathrm{Pic}^0(E_1)\times E_1$; denote by $\pi_i$ the projection onto the $i$-th factor.  Let $\mathcal Q$ be the object $\mathcal P\otimes\pi_2^*(L^{-1})[-n]$ in $T_{\mathrm{Pic}^0(E_1)\times E_1}$.  Then we have an exact equivalence \[T_{\mathrm{Pic}^0(E_1)}\stackrel{\Phi_\mathcal Q}{\longleftarrow} T_{E_1},\]which sends the object $v\cong L\otimes B[n]$ to $\pi_{1*}(\mathcal P\otimes \pi_2^*(B))$ in $T_{\mathrm{Pic}^0(E_1)}$.

We consider the \emph{vector bundle} $\mathcal P\otimes \pi_2^*(B)$ as a family of vector bundles on $E_1$ parametrized by $\mathrm{Pic}^0(E_1)$.  In fact it is the family of vector bundles of rank $r$ and degree $d_B$ on $E_1$ with ``level $r/h$-structure'', where $h=\gcd(r,d_B)$; see \cite[Theorem 10]{atiyah_ellipticvectorbundles}.  

Recall that we have $0<d_B<r$, and so by \cite[Lemma 15]{atiyah_ellipticvectorbundles} and the Rieman-Roch theorem \cite[Lemma 8]{atiyah_ellipticvectorbundles}, we have for every line bundle $M$ of degree zero:\[H^i(M\otimes B)\cong \left\{\begin{array}{ll}k^{d_B} & \mbox{if $i=0$,} \\ 0 & \mbox{if $i\neq 0$.} \end{array}\right.\]

Hence by the Leray spectral sequence \cite[(3.4)]{huybrechts_fm} we see that $v':=\pi_{1*}(\mathcal P\otimes \pi_2^*(B))$ is a vector bundle on $\mathrm{Pic}^0(E_1)$ of rank $d_B<r$, and we can simply take $\Phi'$ to be the composition of a quasi-inverse of $\Phi_\mathcal Q$ followed by $\Phi$:

\centerline{\xymatrix{T_{\mathrm{Pic}^0(E_1)} & T_{E_1} \ar[r]^-{\Phi}_\cong \ar[l]_-{\Phi_\mathcal Q}^-\cong &  T_{E_2} }  } 

Identifying $\mathrm{Pic}^0(E_1)$ with $E_1$ then concludes the proof of the lemma and with it the proof of \rrf{prop:ell}.  \end{proof}

\ppp{par:rmk} \emph{Remarks.}  The same idea above may in fact be used to give a proof of \rrf{thm:curves} in the case when the genus $g$ of $C_1$ is not equal to one.  Here we give a sketch:  Suppose $Y$ is a smooth projective variety and $C$ is a smooth projective curve admitting an exact equivalence \[\Phi: T_C\longrightarrow T_Y.\]

Let $y\in Y$ be a closed $k$-point.  Denote by $v\in T_C$ the shift of a simple coherent sheaf such that $\Phi(v)\cong k(y)[-\dim(Y)]$.  Since Serre functors commute with exact equivalences, we have a commutative diagram 

\centerline{\xymatrix{T_C \ar[r]^-\Phi \ar[d]_-{-\otimes K_C[1]}  & T_Y \ar[d]^-{-\otimes K_Y[\dim(Y)]} \\ T_C \ar[r]^-\Phi & T_Y.}}

Applying the functors in this diagram to the object $v$ gives 

\centerline{\xymatrix{ v \ar@{|->}[r] \ar@{|->}[d]& k(y)[-\dim(Y)] \ar@{|->}[d] \\ v\otimes K_C[1] \ar@{|->}[r] & k(y).}    } 

But this implies that $v[\dim(Y)]$ and $v\otimes K_C[1]$ are sent to isomorphic objects under the equivalence $\Phi$, and this is possible only if these two objects are isomorphic; in particular we have $\dim(Y)=1$.  From this we see that we must have $v\cong v\otimes K_C^{\otimes m}$ for all $m\in \mathbb Z$, but this is possible only if either $v$ is torsion (in which case we conclude $C\cong Y$ by \rrf{cor:bir}) or $K_C$ and $-K_C$ are both not ample (in which case $C$ is of genus one).

\mysubsec{Varieties of maximal Kodaira dimension}

\ppp{}  Consider the following 

\begin{thm}\cite[Theorem 2.3 (2)]{kawamata_dk}  Let $X$ and $Y$ be smooth projective varieties over a field $k$ admitting an exact equivalence $\Phi: T_X\rightarrow T_Y$.  Suppose that the Kodaira dimension of $X$ is equal to $\dim(X)$, then $X$ and $Y$ are $K$-equivalent.
\end{thm}

Here by $K$-equivalence we mean there is another projective smooth variety $Z$ with birational maps $f:Z\rightarrow X$ and $g:Z\rightarrow Y$ so that $f^*K_X\cong g^*K_Y$ on $Z$.  It seems unclear at the moment and an interesting question how to interpret $K$-equivalence in terms of the derived categories.

\begin{cor}\label{cor:DK}  Let $X$ and $Y$ be smooth projective varieties over a field $k$ admitting an exact equivalence $\Phi: T_X\rightarrow T_Y$.  Suppose that the Kodaira dimension of $X$ is equal to $\dim(X)$, then $X$ and $Y$ are birational.
\end{cor}

In the remaining paragraphs of this section we give a different proof of \rrf{cor:DK}, and we work under its notations and assumptions.  An analogous arguments works if we replace the Kodaira dimension of $X$ in the statement with the Kodaira dimension of the anti-canonical bundle $-K_X$.

\ppp{}  Let $y\in Y$ be a closed $k$-point, and $v_y\in T_X$ be an object such that $\Phi(v_y)\cong k(y)[-\dim(Y)]$.  Then as in \rrf{prop:obs} we know that $\mathrm{Hom}(v_y,v_y)\cong k$.  By \cite[Lemma 3.9]{huybrechts_fm} we know that the cohomology support $Z_y:=\mathrm{supp}(v_y)$ is connected.

\begin{lemma}\label{lem:support}  For every point $x\in X$, there is a point $y\in Y$ such that $x$ lies in$Z_y$.
\end{lemma}

\begin{proof}  By \cite[Prpposition 3.17]{huybrechts_fm} we know that $\{k(y)[-\dim(Y)]\}$ (with $y$ varying) form an spanning class of $T_Y$.  Since $\Phi$ is an equivalence we see that $\{v_y\}$ form an spanning class of $T_X$.  In particular for any $x$ we can find an $y$ so that \[x^*(v_y)\cong\mathrm{Hom}^\bullet(k(x)[-\dim(X)],v_y)\neq 0\in T_{k(x)}.\]  \end{proof}

\ppp{}  By \rrf{cor:bir2}, the proof of \rrf{cor:DK} is reduced to the following

\begin{prop}\label{prop:maxkod}  If the Kodaira dimension of $X$ is equal to $\dim(X)$ then there is a point $y\in Y$ with $\dim(Z_y)=0$.  \end{prop}

First we need a simple

\begin{lemma}\label{lem:fixedsheaf}  Let $Z$ be a smooth projective variety over a field $k$, $M$ a line bundle on $Z$ and $\mathcal G$ a coherent sheaf with $\mathrm{supp}(\mathcal G)=Z$ such that $\mathcal G\cong \mathcal G\otimes M$.  Then $M$ is torsion in $\mathrm{Pic}(Z)$; that is, $\mathcal O_Z\cong M^{\otimes r}$ for some $r\in\mathbb N$.  \end{lemma}

\begin{proof}  We may replace $\mathcal G$ in the statement with a reflexive sheaf \cite[Corollary 1.4]{hartshorne_reflexive}:  Indeed, $\mathcal G^\vee{}^\vee$ is reflexive and is easily seen to satisfy $\mathcal G^\vee{}^\vee\cong \mathcal G^\vee{}^\vee\otimes M$.  So from now on we assume $\mathcal G$ is reflexive.

Then there is a closed subset $Z'\subset Z$ of codimension at least two such that the restriction of $\mathcal G$ to its complement $U:=Z-Z'$ is a vector bundle of \emph{positive} rank $r$.  By taking the $r$-th wedge power of $\mathcal G|_U\cong \mathcal G\otimes M|_U$ we get \[\bigwedge^r(\mathcal G|_U)\cong \bigwedge^r(\mathcal G|_U)\otimes (M|_U^{\otimes r}).\]

Since $\bigwedge^r(\mathcal G|_U)$ is a line bundle on $U$, this implies that $M|_U^{\otimes r}$ is the trivial line bundle on $U$, which in turn implies that $M^{\otimes r}$ is the trivial line bundle on $Z$ since $Z'$ has codimension at least two.  \end{proof}

\begin{proof}[Proof of \rrf{prop:maxkod}]  Suppose on the contrary that for every $y$ we have $\dim(Z_y)>0$.  Consider the following commutative diagram

\centerline{\xymatrix{T_X \ar[r]^-\Phi \ar[d]_-{-\otimes K_X[\dim(X)]}  & T_Y \ar[d]^-{-\otimes K_Y[\dim(Y)]} \\ T_X \ar[r]^-\Phi & T_Y,}}

where the vertical arrows are Serre functors.  Applying the functors in this diagram to the object $v_y$ gives 

\centerline{\xymatrix{ v_y \ar@{|->}[r] \ar@{|->}[d]& k(y)[-\dim(Y)] \ar@{|->}[d] \\ v_y\otimes K_X[\dim(X)] \ar@{|->}[r] & k(y).}    }

Then $v_y[\dim(Y)]$ and $v\otimes K_X[\dim(X)]$ are sent to isomorphic objects in $T_Y$ under the equivalence $\Phi$, hence these two objects are isomorphic in $T_X$.  Letting $d=\dim(X)-\dim(Y)$ we then have \[v_y\cong v_y\otimes K_X[d]\cong v_y\otimes K_X^{\otimes 2}[2d]\cong \cdots \cong v_y\otimes K_X^{\otimes m}[md]\cong \cdots\]for all $m\in\mathbb Z$.  This in particular implies the well-known fact that $d=0$ \cite[Proposition 4.1]{huybrechts_fm}, and we have $v_y\cong v_y\otimes K_X^{\otimes m}$ for every $m\in\mathbb Z$.

Fix any closed $k$-point $x$ on $X$ and $y$ on $Y$ so that $x\in Z_y$ as in \rrf{lem:support}.  Let $h:Z\rightarrow Z_y$ be a non-constant morphism from a positive dimensional smooth variety $Z$ whose image contains $x$; for example we can choose any projective curve contained in $Z_y$ passing through $x$, then take $Z$ to be its normalization.  This can be done since $Z_y$ is connected and assumed to be of positive dimension.

The isomorphism $v_y\cong v_y\otimes K_X^{\otimes m}$ then gives $h^*(v_y)\cong h^*(v_y)\otimes h^*(K_X^{\otimes m})$ on $Z$.  Since $h^*(K_X^{\otimes m})$ is locally free on $Z$ this implies \[\mathcal H^i(h^*(v_y))\cong \mathcal H^i(h^*(v_y))\otimes h^*(K_X^{\otimes m}),\]for every $i$; here $\mathcal H^i(-)$ denotes the $i$-th cohomology sheaf.  

Since the image of $h$ is contained in the support of $v_y$, we have $\mathrm{supp}(\mathcal H^i(h^*(v_y)))=Z$ for some $i$.  Therefore the conditions of \rrf{lem:fixedsheaf} are satisfied with $\mathcal G=\mathcal H^i(h^*(v_y))$ and $M=h^*(K_X^{\otimes m})$, and we conclude that $h^*(K_X^{\otimes m})$ is a torsion element in $\mathrm{Pic}(Z)$ for every $m$.  

But this implies that every general $x$ is contained in a positive dimensional fibre of the rational map $\phi_{K_X^{\otimes m}}$ associated to $K_X^{\otimes m}$, since $h(Z)$ is mapped to a point.  Therefore the Kodaira dimension of $X$ cannot be equal to $\dim(X)$.  \end{proof}

\ppp{}  \emph{Remark.}  Denote by $\kappa(X,L)$ the Iitaka-Kodaira dimension of a line bundle on a smooth projective variety $X$.  Then the proof of \rrf{prop:maxkod} shows that if $v\in T_X$ is an object that is sent to a shift of $k(y)$ under an exact equivalence $\Phi: T_X\rightarrow T_Y$, then the support of $v$ is contracted by the rational maps $\phi_{\pm K_X}$ associated to the linear systems $|\pm K_X|$.  From this we have \[\dim(X)-\kappa(X,\pm K_X)\geq \dim\mathrm{supp}(v).\]

\bibliographystyle{amsalpha}
\bibliography{mybibli}

\providecommand{\bysame}{\leavevmode\hbox to3em{\hrulefill}\thinspace}
\providecommand{\MR}{\relax\ifhmode\unskip\space\fi MR }
\providecommand{\MRhref}[2]{%
  \href{http://www.ams.org/mathscinet-getitem?mr=#1}{#2}
}
\providecommand{\href}[2]{#2}
\begin{thebibliography}{Kaw10}

\bibitem[Ati57]{atiyah_ellipticvectorbundles}
M.~F. Atiyah, \emph{Vector bundles over an elliptic curve}, Proc. London Math.
  Soc. (3) \textbf{7} (1957), 414--452.

\bibitem[BM01]{bm_complexsurfaces}
Tom Bridgeland and Antony Maciocia, \emph{Complex surfaces with equivalent
  derived categories}, Math. Z. \textbf{236} (2001), no.~4, 677--697.

\bibitem[BO95]{bo_sod}
A.~I. Bondal and D.~O. Orlov, \emph{Semiorthogonal decomposition for algebraic
  varieties}, arXiv:alg-geom/9506012v1, 1995.

\bibitem[BO01]{BO_reconstruction}
Alexei Bondal and Dmitri Orlov, \emph{Reconstruction of a variety from the
  derived category and groups of autoequivalences}, Compositio Math.
  \textbf{125} (2001), no.~3, 327--344.

\bibitem[Har80]{hartshorne_reflexive}
Robin Hartshorne, \emph{Stable reflexive sheaves}, Math. Ann. \textbf{254}
  (1980), no.~2, 121--176.

\bibitem[Huy06]{huybrechts_fm}
D.~Huybrechts, \emph{Fourier-{M}ukai transforms in algebraic geometry}, Oxford
  Mathematical Monographs, The Clarendon Press Oxford University Press, Oxford,
  2006.

\bibitem[Kaw02]{kawamata_dk}
Yujiro Kawamata, \emph{{$D$}-equivalence and {$K$}-equivalence}, J.
  Differential Geom. \textbf{61} (2002), no.~1, 147--171. \MR{1949787
  (2004m:14025)}

\bibitem[Kaw10]{kawamata_kd}
\bysame, \emph{Derived categories and minimal models}, Sugaku Expositions
  \textbf{23} (2010), no.~2, 235--259, Sugaku Expositions.

\bibitem[Liu11]{liu_TTsp}
Yu-Han Liu, \emph{Functors on triangulated tensor categories}, preprint,
  arXiv:1105.2197v1, 2011.

\bibitem[Orl03]{orlov_dercat}
D.~O. Orlov, \emph{Derived categories of coherent sheaves and equivalences
  between them}, Uspekhi Mat. Nauk \textbf{58} (2003), no.~3(351), 89--172.

\bibitem[Tho97]{thomason_class}
R.~W. Thomason, \emph{The classification of triangulated subcategories},
  Compositio Math. \textbf{105} (1997), no.~1, 1--27.

\bibitem[Tu93]{tu_elliptic}
Loring~W. Tu, \emph{Semistable bundles over an elliptic curve}, Adv. Math.
  \textbf{98} (1993), no.~1, 1--26.

\end{thebibliography}

 \end{document}